\documentclass[12pt, reqno]{amsart}

\usepackage[height=22.5cm, width=16.5cm, hmarginratio={1:1}]{geometry}
\usepackage{amssymb}
\usepackage{amsmath}
\usepackage[OT2,T1]{fontenc}
\usepackage{graphicx,color}
\usepackage{wrapfig,framed}
\usepackage{amsfonts}
\usepackage{array}
\usepackage{geometry}
\usepackage{amsthm}
\usepackage{xypic}
\usepackage{tabularx}
\usepackage{hyperref}
\usepackage{tikz}
\newcommand*{\sheafhom}{\mathcal{H}\kern -.5pt om}

\usepackage[maxbibnames=99]{biblatex}

\renewbibmacro{in:}{}
\DeclareFieldFormat[article]{pages}{#1}
\DeclareFieldFormat[incollection]{pages}{#1}
\DeclareFieldFormat[inbook]{citetitle}{#1}
\DeclareFieldFormat[inbook]{title}{#1} 
\DeclareFieldFormat[article]{citetitle}{#1}
\DeclareFieldFormat[article]{title}{#1} 
\DeclareFieldFormat[article]{number}{} 
\DeclareFieldFormat[incollection]{title}{#1} 
\DeclareFieldFormat[incollection]{booktitle}{#1} 
\DeclareFieldFormat{journaltitle}{#1\isdot}
\DeclareFieldFormat{title}{#1\isdot}
\DeclareFieldFormat[article]{volume}{\mkbibbold{#1}}
\DeclareFieldFormat[inbook]{volume}{Vol.\mkbibbold{#1}}
\DeclareNameAlias{default}{last-first}

\DefineBibliographyExtras{english}{%
}

\theoremstyle{plain}
\newtheorem{theorem}{Theorem}
\newtheorem{proposition}{Proposition}
\newtheorem{lemma}{Lemma}

\theoremstyle{definition}

\newtheorem{conjecture}{Conjecture}

\newtheorem{remark}{Remark}

\newlanguage\fakelanguage
\newcommand\cyr{\fontencoding{OT2}\fontfamily{wncyr}\selectfont
   \language\fakelanguage}
\DeclareTextFontCommand{\textcyr}{\cyr}

\begin{document}

\title[Refined Dubrovin conjecture for the Lagrangian Grassmanian $LG(2,4)$]{Proof of the refined Dubrovin conjecture for the Lagrangian Grassmanian $LG(2,4)$}
\author{Fangze Sheng}
\address{School of Mathematical Sciences, University of Science and Technology of China, Hefei 230026, P. R. China}
\email{fzsheng2000@mail.ustc.edu.cn}

\begin{abstract} 
The Dubrovin conjecture predicts a relationship between the monodromy data of the Frobenius manifold associated to the quantum cohomology of a smooth projective variety and the bounded derived category of the same variety. A refinement of this conjecture was given by Cotti, Dubrovin and Guzzetti, which is equivalent to the Gamma conjecture II proposed by Galkin, Golyshev and Iritani. The Gamma conjecture II for quadrics was proved by Hu and Ke. The Lagrangian Grassmanian $LG(2,4)$ is isomorphic to the quadric in $\mathbb P^4$. In this paper, we give a new proof of the refined Dubrovin conjecture for the Lagrangian Grassmanian $LG(2,4)$ by explicit computation. 
\end{abstract}
\maketitle

\setcounter{tocdepth}{1}
\tableofcontents

\section{Introduction}

The monodromy data of a semisimple Dubrovin--Frobenius manifold $M$ \cite{Dubrovin2dtft,dubrovin1999painleve} consist of four constant matrices $(\mu, R, S, C)$. Here, $\mu$ and $R$ are given by monodromy data of the ODE system associated to {\it the extended Dubrovin connection} on $M \times \mathbb C^{*}$ at $z=0$, the matrix $S$, called the {\it Stokes matrix}, describes asymptotic behaviours of solutions to the same system as $|z| \rightarrow \infty$, and the matrix $C$, called the {\it central connection matrix}, relates two canonically-constructed analytic solutions. Dubrovin proved that one can reconstruct the Dubrovin--Frobenius structure of a semisimple Dubrovin--Frobenius manifold from its monodromy data $(\mu, R, S, C)$ \cite{Dubrovin2dtft,dubrovin1999painleve}. 

Let $X$ be a $D$-dimensional smooth projective variety with vanishing odd cohomology, and $X_{0,k,\beta}$ the moduli space of stable maps of degree $\beta\in H_2(X;\mathbb{Z})/{\rm torsion}$ with target~$X$ from curves of genus~$0$ with~$k$ distinct marked points.  Choose a homogeneous basis $\phi_1 = 1,\phi_2,\dots,\phi_n$ of the cohomology ring $H^{*}(X;\mathbb{C})$ such that $\phi_\alpha \in H^{2q_\alpha}(X;\mathbb C)$, $\alpha=1,\dots,n$,  $0=q_1<q_2 \leq \dots \leq q_{n-1} < q_n = D$. The {\it Gromov--Witten (GW) invariants of~$X$ of genus~$0$ and degree~$\beta$} are the integrals \cite{Behrend1997,kontsevich1994gromov,Dubrovin2dtft,li1998virtual,manin1999frobenius,ruan1995mathematical}
\begin{equation}
   \int_{[X_{0,k,\beta}]^{\rm virt}} {\rm ev}_1^*(\phi_{\alpha_1}) \cdots  {\rm ev}_k^*(\phi_{\alpha_k}), \quad 
\alpha_1,\dots,\alpha_k=1,\dots,n. 
\end{equation} 
Here, ${\rm ev}_a$, $a=1,\dots,k$, are the evaluation maps, and $[X_{0,k,\beta}]^{\rm virt}$ is the virtual fundamental class. The {\it genus~$0$ primary free energy $F_0({\bf v};Q)$} for the GW invariants of~$X$ is defined by
\begin{equation}
F_0({\bf v};Q) = \sum_{k\ge0}\sum_{\alpha_1,\dots,\alpha_k \ge0} \sum_{\beta}\frac{Q^\beta v^{\alpha_1} \cdots v^{\alpha_k} }{k!} 
\int_{[X_{0,k,\beta}]^{\rm virt}} {\rm ev}_1^*(\phi_{\alpha_1}) \cdots {\rm ev}_k^*(\phi_{\alpha_k}),
\end{equation}
where $(v^1,\dots,v^n)$ are indeterminates, and 
\begin{equation}
Q^\beta=Q_1^{m_1} \cdots Q_r^{m_r} \quad  ({\rm for}~\beta=m_1\beta_1+\dots+m_r\beta_r)
\end{equation}
is an element of the Novikov ring $\mathbb C [[Q_1,\dots,Q_r]]$. 
Here, $(\beta_1,\dots,\beta_r)$ is a basis of $H_2(X;\mathbb{Z})/{\rm torsion}.$ When $X$ is Fano, it is known that $F_0({\bf v};Q)$ is a power series of $v^1, Q_1 e^{v^2}, \dots, Q_r e^{v^{r+1}},$ $ v^{r+2}, \dots, v^n$ and one can 
take $Q_1=\cdots=Q_r=1$ (cf. \cite{Dubrovin2dtft, manin1999frobenius}). If $F_0({\bf v}):=F_0({\bf v}; Q)|_{Q_1=\cdots=Q_r=1}$ is convergent, then an analytic Dubrovin--Frobenius structure can be defined on its convergence domain. The {\it quantum cohomology of $X$} is the commutative ring given by flat vector fields on this Dubrovin--Frobenius manifold, and its restriction to $H^2(X;\mathbb C)$ is called the {\it small quantum cohomology of $X$}. The quantum cohomology of $X$ carries a canonical {\it calibration} \cite{Dubrovin2dtft,Givental2001GWquantizationof} obtained from genus 0 GW invariants, which can also be fixed by using an analytic criterion \cite{Galkin2014gammaclasses} (cf. \cite{cotti2018helix}).

In his 1998 ICM talk \cite{Dub98Geometry}, Dubrovin proposed a conjecture on the relationship between monodromy data of the quantum cohomology of a smooth projective variety and the bounded derived category of the same variety. We refer the readers to Section \ref{sec:conjecture} for a precise statement of the conjecture (see Conjecture \ref{originalconjecture}) and for a refinement proposed by Cotti, Dubrovin and Guzzetti \cite{cotti2018helix} (see Conjecture \ref{refinedconjecture}). Note that part 3 of the refined Dubrovin conjecture \ref{refinedconjecture} is equivalent \cite{cotti2018helix} to the Gamma conjecture II proposed by Galkin, Golyshev and Iritani \cite{Galkin2014gammaclasses}. Recently, some new significance of the Dubrovin conjecture (and its refined version) is given in \cite{yang2023analytic}. 

In the past two decades, the Dubrovin conjecture has been proved for some cases. Dubrovin proposed a method to directly compute the monodromy data in \cite{dubrovin1999painleve}, where the Stokes matrix for the quantum cohomology of $\mathbb{P}^2$ was computed. In \cite{Dub98Geometry}, Dubrovin computed the central connection matrix for $\mathbb{P}^2$, and proved the Dubrovin conjecture for this case. The Stokes matrix for an arbitrary complex projective space was computed by Guzzetti in \cite{guzzetti1999stokes}, and part 2 of Conjecture \ref{originalconjecture} was proved (part 1 of Conjecture \ref{originalconjecture} for this case was known). The same results were also obtained by Tanab{\'e} in \cite{Tanabe2004Invariantprojective}. 
Recently, the refined Dubrovin conjecture for an arbitrary Grassmanian was proved by Cotti, Dubrovin and Guzzetti \cite{cotti2018helix} (cf. \cite{Cotti2020LocalModuli,Ueda2005Grassmanian}). 
In \cite{cotti2022cyclicHirzebruch}, Cotti computed the monodromy data and then proved the refined conjecture for Hirzebruch surfaces. Part 2 of Conjecture \ref{originalconjecture} has been verified for other cases, including cubic surfaces by Ueda \cite{ueda2005stokescubic}, minimal Fano threefolds by Golyshev \cite{Golyshev2009Minimalexceptional}, a class of orbifold projective lines by Iwaki and Takahashi \cite{IwakiTakahashi2013stokesorbifold}, weighted projective spaces and certain smooth Fano hypersurfaces by Cruz Morales and van der Put \cite{MoralesPut2015stokesFano}. The Gamma conjecture II has been proved for arbitrary Grassmanians in \cite{Galkin2014gammaclasses} (cf. \cite{cotti2018helix}), for quadrics of arbitrary dimensions by Hu and Ke \cite{HuKe2023gammaquadrics} and for certain toric Fano manifolds by Fang and Zhou \cite{fang2019gammaiitoricvarieties}.


Let $V$ be the vector space $\mathbb{C}^{2n}$, equipped with a symplectic form. A subspace $\Sigma$ of $V$ is called {\it isotropic} if the restriction of the symplectic form to $\Sigma$ vanishes. The maximal possible dimension of an isotropic subspace is $n$, and in this case $\Sigma$ is called a {\it Lagrangian subspace}. The {\it Lagrangian Grassmanian} $LG(n,2n)$ is the projective complex manifold which parameterizes Lagrangian subspaces $\Sigma$ in $V$. A description of the small quantum cohomology of $LG(n,2n)$ was formulated by Kresch and Tamvakis by quantum Schubert calculus \cite{kresch2003quantum} (cf. \cite{SiebertTian1994OnQC}). The existence of a {\it full exceptional collection} in the bounded derived category of coherent sheaves on a Lagrangian Grassmanian was proved for $LG(3,6)$ by Samokhin \cite{Samokhin2001LG3C}, for $LG(4,8)$ by Polishchuk and Samokhin \cite{PolishchukSamokhin2011LG48510}, and for $LG(n,2n)$, $n \geq 1$ by Fonarev \cite{fonarev2022FECLagrangian}.

The main result of this paper is the following theorem. 
\begin{theorem}[= {\it Theorem \ref{maintheorem}}]\label{maintheorem2}
    The refined Dubrovin conjecture holds for the Lagrangian Grassmanian $LG(2,4)$. 
\end{theorem}

Since $LG(2,4)$ is isomorphic to the quadric in $\mathbb P^4$, Theorem \ref{maintheorem2} can also be deduced from the corollary of Hu-Ke \cite{HuKe2023gammaquadrics}. In this paper, we will give an explicit proof by computing the monodromy data for the quantum cohomology of $LG(2,4)$. 

The rest of the paper is organized as follows. In Section \ref{sec:2}, we review the monodromy data of semisimple Dubrovin--Frobenius manifolds. In Section \ref{sec:conjecture}, we recall some basic notions of derived category, and state the Dubrovin conjecture and a refined version. In Section \ref{sec:4}, we briefly review some results of the small quantum cohomology of $LG(2,4)$ and then give our proof of Theorem \ref{maintheorem2}.

\section{Review on monodromy data of semisimple Dubrovin--Frobenius manifolds}\label{sec:2}

A {\it Frobenius algebra} is a triple $(A, \, e, \, \langle\,,\,\rangle)$, where $A$ is a commutative and associative algebra over $\mathbb{C}$ with unity $e$, and $\langle\,,\,\rangle: A\times A\rightarrow \mathbb{C}$ is a symmetric, non-degenerate and bilinear product satisfying 
$\langle x\cdot y , z\rangle = \langle x, y\cdot z\rangle$, $\forall\,x,y,z \in A$. 
A {\it Dubrovin--Frobenius structure of charge~$D$} \cite{Dubrovin2dtft} on a complex manifold~$M$ is a family of 
Frobenius algebras $(T_p M, \, e_p, \, \langle\,,\,\rangle)_{p\in M}$, 
depending holomorphically on~$p$ and satisfying the following three axioms:
\begin{itemize}
\item[{1}]  The metric $\langle\,,\,\rangle$ is flat; moreover, 
denote by~$\nabla$ the Levi--Civita connection of~$\langle\,,\,\rangle$, then it is required that $\nabla e = 0$. 
\item[{2}] Define a $3$-tensor field~$c$ by 
$c(X,Y,Z):=\langle X\cdot Y,Z\rangle$, for $X,Y,Z$ being holomorphic vector fields on~$M$. 
Then the $4$-tensor field $\nabla c$ is required to be 
symmetric. 
\item[{3}]
There exists a holomorphic vector field~$E$ on~$M$ satisfying 
\begin{align}
&  [E, X\cdot Y]-[E,X]\cdot Y - X\cdot [E,Y] = X\cdot Y, \label{E2}\\
&  E \langle X,Y \rangle - \langle [E,X],Y\rangle - \langle X, [E,Y]\rangle 
= (2-D) \langle X,Y\rangle. \label{E3}
\end{align}
\end{itemize}
A complex manifold endowed with a Dubrovin--Frobenius structure of charge~$D$ is  
called a {\it Dubrovin--Frobenius manifold of charge~$D$}, with $\langle \, ,\, \rangle$ 
being called the {\it invariant flat metric} and $E$ the {\it Euler vector field}.
We denote the Dubrovin--Frobenius manifold by $(M,\, \cdot,\, e,\, \langle \,,\, \rangle,\, E)$. 
A point~$p$ on a Dubrovin--Frobenius manifold~$M$ is called a {\it semisimple point}, if $T_pM$ is semisimple. 
A Dubrovin--Frobenius manifold is called {\it semisimple}, if its generic points are semisimple points.

For an $n$-dimensional Dubrovin--Frobenius manifold $(M, \, \cdot, \, e, \, \langle\,,\,\rangle,\ E)$, an flat affine connection $\tilde{\nabla}$, called {\it extended Dubrovin connection}, can be defined \cite{Dubrovin2dtft} on $M\times \mathbb{C}^*$ through 
    \begin{equation}
    \begin{aligned}
        \tilde{\nabla}_X Y&:=\nabla_X Y + z X \cdot Y,
        \\ \tilde{\nabla}_{\partial_z} Y&:= \dfrac{\partial Y}{\partial z} + E \cdot Y- \dfrac{1}{z}\left(\frac{2-D}2 {\rm id} - \nabla E\right)Y,
        \\ \tilde{\nabla}_{\partial_z} \partial_z&:=0, \qquad \tilde{\nabla}_{X} \partial_z :=0.
    \end{aligned}
    \end{equation}
Here, $X,Y$ are arbitrary holomorphic vector fields on $M \times \mathbb{C}^*$ with 
zero component along $ \frac{\partial}{\partial z}$.
Choose flat coordinates ${\bf v}=(v^1,\dots,v^n)$ such that $E=\sum_{\beta=1}^n \left(\left(1-\frac{D}2-\mu_\beta\right) v^\beta + r^\beta\right) \frac{\partial}{\partial v^\beta}$. The differential system for a flat section $y=(y^1,\dots,y^n)^T$ reads (see \cite{Dubrovin2dtft, dubrovin1999painleve})
\begin{align}
& \frac{\partial y}{\partial v^\alpha} = z C_\alpha y, \quad \alpha=1,\dots,n, \label{pdedubrovin}\\
& \frac{d y}{d z} = \left(\mathcal{U}+\frac{\mu}{z} \right) y, \label{odedubrovin}     
\end{align}
where $C_{\alpha}$ is the matrix of multiplication by $\frac{\partial}{\partial v^\alpha}$, $\mathcal{U}$ is the matrix of multiplication by the Euler vector field and $\mu = \text{diag}(\mu_1,\dots,\mu_n)$.

Fix $\bf v$. Then there exists a fundamental matrix solution to \eqref{odedubrovin} of the form \cite{Dubrovin2dtft, dubrovin1999painleve} 
\begin{equation}\label{dubrovinleveltform}
\mathcal{Y}_0(z) = \Phi(z) z^\mu z^R,
\end{equation}
where $R$ is a constant matrix, $\Phi(z)$ is an analytic matrix-valued function on $\mathbb{C}$ satisfying 
\begin{equation}
\Phi(0) = I, \quad \Phi(-z)^T\eta\Phi(z)= \eta.
\end{equation}
The matrices $\mu,R$ are the {\it monodromy data at $z=0$ of the Dubrovin--Frobenius manifold} \cite{Dubrovin2dtft, dubrovin1999painleve}. 

Now we restrict our consideration to a semisimple point $p\in M$.  Denote the restriction of coordinate vector fields $(\frac{\partial}{\partial v^1},\dots,\frac{\partial}{\partial v^n})$ to this point by $(e_1,\dots,e_n)$. Let $u_1,\dots, u_n$ be eigenvalues to $\mathcal{U}$ and assume they are pairwise distinct. Let $f_1,\dots,f_n$ be the normalized eigenvectors such that
\begin{equation}
 f_i \cdot f_j = \delta_{ij} f_j\ \text{and}\ \langle  f_i ,  f_j \rangle = \delta_{ij}, \qquad i,j=1,\dots,n.
\end{equation}
Let $\Psi=(\psi_{i\alpha})$ be the transition matrix from the orthonormal basis $(f_1,\dots,f_n)$ to the basis $(e_1,\dots,e_n)$, i.e.,
\begin{equation}\label{psitransition}
e_{\alpha}=\sum_{i=1}^n \psi_{i\alpha} f_i , \quad \alpha=1,\dots,n.
\end{equation} 
For any fundamental matrix solution $\mathcal{Y}(z)$ to \eqref{odedubrovin}, let $Y(z)=\Psi \mathcal{Y}(z). $ Then $Y=Y(z)$ satisfies the equation 
\begin{align}
&\frac{dY}{dz} = \biggl(U+\frac{V}z\biggr) Y \label{Yeqcanonical}, 
\end{align}
where $U = \Psi \mathcal{U} \Psi^{-1} = {\rm diag}(u_1,\dots,u_n) $ and $ V=\Psi \mu \Psi^{-1}$. 

Note that $z=\infty$ is an irregular singularity of~\eqref{Yeqcanonical} 
of Poincar\'e rank~1. Then \cite{Dubrovin2dtft, dubrovin1999painleve} (cf.~\cite{Cotti2019Isomonodromydeformations, Cotti2020LocalModuli}) 
 equation \eqref{Yeqcanonical} has a unique formal solution $Y_{\rm formal}(z)$ of the form
\begin{equation}\label{asymptoticformal}
Y_{\rm formal}(z) = G (z) e^{z U},
\end{equation}
where $G (z)$ is a formal power series of~$z^{-1}$ with $G(\infty)=I$ and satisfies the orthogonality condition $G (-z)^T G (z) = I$. 

A line $\ell$ through the origin in the complex $z$-plane is called {\it admissible} if
\begin{equation}
{\rm Re} \, z(u_i- u_j)|_{z \in \ell \setminus 0} \neq 0, \quad \forall\, i\neq j.
\end{equation}
Fix any admissible line $\ell$ with an orientation. Denote by $\phi\in[0,2\pi)$ the angle from the positive real axis and to the positive direction of~$\ell$. 
Construct two sectors
\begin{align}
\Lambda_{\rm left}: \phi-\varepsilon<{\rm arg} \, z< \phi + \pi+ \varepsilon, \label{tildesectorL}\\ 
\Lambda_{\rm right}: \phi-\pi-\varepsilon<{\rm arg} \, z< \phi + \varepsilon, \label{tildesectorR} 
\end{align}
where $\varepsilon$ is a sufficiently small positive number. From~\cite{BalserJurkatLutz1979Birkhoff,BalserJurkatLutz1981OnTheReduction, Dubrovin2dtft, dubrovin1999painleve,guzzetti2021laplacetransform}(cf. \cite{guzzetti2016stokescoefficients})
there exist unique fundamental matrix solutions $Y_{L/R}(z)$ to the ODE system~\eqref{Yeqcanonical}, analytic in $\Pi_{{\rm left}/{\rm right}}$, such that 
\begin{equation}\label{YLRasymptotic}
Y_{L/R}(z) \sim Y_{\rm formal}(z)
\end{equation}
as $|z|\to\infty$ within the sectors. 
We will consider the behaviour of the matrix solutions on the narrow sectors
\begin{equation}
\Lambda_+: \phi-\varepsilon < {\rm arg} z < \phi+\varepsilon,
\end{equation}
and 
\begin{equation}
\Lambda_-: \phi-\pi-\varepsilon < {\rm arg} z < \phi-\pi+\varepsilon. 
\end{equation}

In fact, the sectors $\Lambda_{{\rm left}/{\rm right}}$ where the asymptotic behaviour \eqref{YLRasymptotic} holds can be extended. 
We introduce \emph{Stokes rays}, which are oriented rays $R_{ij}$ in $\mathbb C$ defined by
\begin{equation}\label{stokesrays}
    R_{ij}:=\left\{-i(\overline{u_i}-\overline{u_j})\rho\colon \rho \geq 0\right\}. 
\end{equation}
By definition, an admissible line $\ell$ must contain no Stokes rays. The sectors $\Lambda_{{\rm left}/{\rm right}}$ can be extended up to the first nearest Stokes ray (see \cite{BalserJurkatLutz1979Birkhoff}). We denote the extended sectors by $\Pi_{{\rm left}/{\rm right}}$ and the corresponding extended narrow sectors by $\Pi_+$ and $\Pi_{-}$. 
Within the narrow sector $\Pi_+$, the two analytic fundamental matrix solutions in \eqref{YLRasymptotic}  is related by a constant matrix $S$ called the {\it Stokes matrix with respect to the 
admissible line~$\ell$} (cf. \cite{Dubrovin2dtft, dubrovin1999painleve}):
\begin{equation}
Y_L(z) =Y_R(z) S, \quad z\in \Pi_+.
\end{equation} 
Similarly, according to \cite{Dubrovin2dtft, dubrovin1999painleve} (cf. \cite{ BalserJurkatLutz1981OnTheReduction}), in the narrow sector $\Pi_-$, we have
\begin{equation}
Y_L(z) =Y_R(z) S^T, \quad z\in \Pi_-.
\end{equation} 
Within the narrow sector $\Pi_+$, the fundamental matrix solutions 
$Y_{0}(z):=\Psi \mathcal{Y}(z)$ and $Y_{R}(z)$ 
to the ODE system~\eqref{Yeqcanonical} also be related by a constant matrix~$C$ called the {\it central connection matrix with respect to the 
admissible line~$\ell$}: 
\begin{equation}
Y_R(z) = Y_{0}(z) C.
\end{equation}
When $M$ is semisimple, the quadruple $(\mu, R, S, C)$ is called the {\it monodromy data of $M$ \cite{dubrovin1999painleve}}.


\begin{proposition}[\cite{Dubrovin2dtft,dubrovin1999painleve}]
$S$ and $C$ satisfy the following restrictions:
\begin{equation}
    C S^T S^{-1} C^{-1} = e^{2\pi i \mu} e^{2\pi i R}, 
\end{equation}
\begin{equation}
    S = C^{-1} e^{-\pi i R} e^{-\pi i \mu} \eta^{-1} (C^T)^{-1}. 
\end{equation}  
\end{proposition}

Following Dubrovin \cite{Dubrovin2dtft,dubrovin1999painleve}, we consider an action of the braid group $\mathcal B_n$ on the set of monodromy data $(S, C)$. Denote by $\beta_{1,2},\dots, \beta_{n-1,n}$ the generators of $\mathcal B_n$, satisfying the relations 
\begin{align*}
   \beta_{i,i+1}\beta_{i+1,i+2}\beta_{i,i+1}=\beta_{i+1,i+2}\beta_{i,i+1}\beta_{i+1,i+2},\\ \beta_{i,i+1}\beta_{j,j+1}=\beta_{j,j+1}\beta_{i,i+1},\quad
\text{if}\ \ |i-j|>1. 
\end{align*}
Let $\ell$ be an admissible line and $(S, C)$ be the monodromy data with respect to $\ell$. Choose the order of canonical coordinates $(u_1,\dots,u_n)$ such that $S$ is an upper triangular matrix. The action of the elementary braid $\beta_{i,i+1}\in\mathcal B_{N}$, $1\leq i \leq n-1$ has the form \cite{Dubrovin2dtft,dubrovin1999painleve}
\begin{equation}
\beta_{i,i+1}(S):= K^{\beta_{i,i+1}}(S)~ S~ K^{\beta_{i,i+1}}(S),
\quad \beta_{i,i+1}(C):= C~ (K^{\beta_{i,i+1}}(S))^{-1},
\end{equation}
where
\begin{equation}
\left(K^{\beta_{i,i+1}}(S)\right)_{kk}=1,\quad\quad k=1,\dots, N\quad k\neq i,i+1,
\end{equation}
\begin{equation}\left(K^{\beta_{i,i+1}}(S)\right)_{i+1,i+1}=-s_{i,i+1},
\end{equation}
\begin{equation}
\left(K^{\beta_{i,i+1}}(S)\right)_{i,i+1}=\left(K^{\beta_{i,i+1}}(S)\right)_{i+1,i}=1, 
\end{equation}
and all other entries of the matrix $K^{\beta_{i,i+1}}$ are equal to zero.

An important class of Dubrovin--Frobenius manifolds come from quantum cohomology. To be precise, let $X$ be a $D$-dimensional smooth projective variety with vanishing odd cohomology, and $X_{g,k,\beta}$ the moduli space of stable maps of degree $\beta\in H_2(X;\mathbb{Z})/{\rm torsion}$ with target~$X$ from curves of genus~$g$ with~$k$ distinct marked points. Here, $g,k\ge0$. We denote the Poincar\'e pairing on 
$H^{*}(X;\mathbb{C})$ by $\langle\,,\,\rangle$. Choose a homogeneous basis $\phi_1 = 1,\phi_2,\dots,\phi_n$ of the cohomology ring $H^{*}(X;\mathbb{C})$ such that $\phi_\alpha \in H^{2q_\alpha}(X;\mathbb C)$, $\alpha=1,\dots,n$, 
 $0=q_1<q_2 \leq \dots \leq q_{n-1} < q_n = d$. 
The integrals
\begin{equation}
   \int_{[X_{g,k,\beta}]^{\rm virt}} c_1(\mathcal{L}_1)^{i_1} {\rm ev}_1^*(\phi_{\alpha_1}) \cdots c_1(\mathcal{L}_k)^{i_k} {\rm ev}_k^*(\phi_{\alpha_k}), \quad 
\alpha_1,\dots,\alpha_k=1,\dots,l, \, i_1,\dots,i_k\geq0, 
\end{equation}
are called {\it Gromov--Witten (GW) invariants of $X$ of genus $g$ and degree~$\beta$} \cite{Behrend1997,kontsevich1994gromov,Dubrovin2dtft,li1998virtual,manin1999frobenius,ruan1995mathematical}. Here, ${\rm ev}_a$, $a=1,\dots,k$ are the evaluation maps, $\mathcal{L}_a$ is the $a$th tautological line bundle on $X_{g,k,\beta}$, 
and $[X_{g,k,\beta}]^{\rm virt}$ is the virtual fundamental class.
The {\it genus~$g$ primary free energy $F_g({\bf v};Q)$} for the GW invariants of~$X$ is defined by 
\begin{equation}\label{Fgenusg}
F_g({\bf v};Q) = \sum_{k\ge0}\sum_{\alpha_1,\dots,\alpha_k \ge0} \sum_{\beta}\frac{Q^\beta v^{\alpha_1} \cdots v^{\alpha_k} }{k!} 
\int_{[X_{g,k,\beta}]^{\rm virt}} {\rm ev}_1^*(\phi_{\alpha_1}) \cdots {\rm ev}_k^*(\phi_{\alpha_k}),
\end{equation}
where $(v^1,\dots,v^n)$ are indeterminates, and 
\begin{equation}
Q^\beta=Q_1^{m_1} \cdots Q_r^{m_r} \quad ({\rm for}~\beta=m_1\beta_1+\dots+m_r\beta_r)
\end{equation}
is an element of the Novikov ring $\mathbb C [[Q_1,\dots,Q_r]]$. Here, $(\beta_1,\dots,\beta_r)$ is a basis of $H_2(X;\mathbb{Z})/{\rm torsion}$. 

Now we restrict our consideration to the case that $X$ is Fano. Then for any fixed $k\ge0$ and $\alpha_1,\dots,\alpha_k \in \{1,\dots,n\}$, 
the sum $\sum_\beta$ in \eqref{Fgenusg} is a finite sum. Moreover, by the divisor equation \cite{Dubrovin2dtft, manin1999frobenius}
\begin{equation}
        \int_{[X_{g,k+1,\beta}]^{\text{vir}}}  \text{ev}_1^*(\phi_{\alpha_1})\cdots \text{ev}_k^*(\phi_{\alpha_k})\phi=\left(\int_{\beta}\phi\right) \int_{[X_{g,k,\beta}]^{\text{vir}}}  \text{ev}_1^*(\phi_{\alpha_1})\cdots \text{ev}_k^*(\phi_{\alpha_k}). 
\end{equation} 
we know that 
$F_0({\bf v};Q)$ is a power series of $v^1, Q_1 e^{v^2}, \dots, Q_r e^{v^{r+1}}, v^{r+2}, \dots, v^n$. Take $Q_1=\cdots=Q_r=1$. 
Denote 
\begin{equation}\label{F0v}
    F_0({\bf v}):=F_0({\bf v}; Q)|_{Q_1=\cdots=Q_r=1}. 
\end{equation}
Assume that the power series $F_0({\bf v})$ has a convergence domain $\Omega$. Then $F_0({\bf v})$ leads to an analytic Dubrovin--Frobenius structure of charge~$D$ on $\Omega$ \cite{kontsevich1994gromov,ruan1995mathematical,Dubrovin2dtft}, with the invariant flat metric $\langle \,,\, \rangle$, the unity vector field $e=\frac{\partial}{\partial v^1}$ and the Euler vector field~$E$ given by 
\begin{equation}
     E=c_1(X)+\sum_{\alpha=1}^n\left(1-\dfrac{1}{2}\text{deg}\phi_\alpha\right)v^{\alpha}\phi_{\alpha},
\end{equation}
where $c_1(X)$ is the first Chern class of $X$. We call the commutative ring given by flat vector fields on the Dubrovin--Frobenius manifold by the {\it quantum cohomology of $X$}, denoted by $QH(X)$, and the restriction of this ring to $\Omega \cap H^2(X;\mathbb C)$ (i.e. points with coordinates $v^i=0$ unless $i=2,\dots, r+1$) the {\it small quantum cohomology of $X$}. Note that the constant matrix $R$ in \eqref{dubrovinleveltform} can be chosen to correspond to  $c_1(X) \cup \colon H^*(X;\mathbb C)\to H^*(X;\mathbb C)$ \cite{dubrovin1999painleve}.

For the quantum cohomology of a smooth projective variety $X$, a canonical {\it calibration} \cite{Dubrovin2dtft,Givental2001GWquantizationof}  obtained from genus 0 GW invariants can be chosen.

\begin{proposition}[\cite{Dubrovin2dtft, dubrovin1999painleve, Galkin2014gammaclasses, Cotti2020LocalModuli, cotti2018helix}]\label{topologicalenumerativeprop}
Fix a point whose coordinates are ${\bf v}=(v^1,\dots,v^n)$ on the quantum cohomoloy of a smooth projective variety $X$. Then a calibration at this point is given by
\begin{align}
&\mathcal Y_{\rm top}(z):=\Phi_{\rm top}(z) z^\mu z^{R},\label{topenumerativesolution} \\
\Phi_{\rm top}(z)_\lambda^\gamma:&=\delta_\lambda^\gamma+\sum_{k,l\geq 0}\sum_{\beta}\sum_{\alpha_1,\dots,\alpha_k}\frac{h_{\lambda,k,l,\beta,\underline\alpha}^\gamma}{k!} v^{\alpha_1}\dots v^{\alpha_k} z^{l+1}. 
\end{align}
Here
\begin{equation}
    h_{\lambda,k,l,\beta,\underline\alpha}^\gamma:=\sum_{\epsilon}\eta^{\epsilon\gamma}\int_{[X_{0,k+2,\beta}]^\text{virt}}c_1(\mathcal L_1)^l {\rm ev}_1^*(\phi_\lambda) {\rm ev}_2^*(\phi_{\epsilon})\prod_{j=1}^k{\rm ev}^*_{j+2}(\phi_{\alpha_j}). 
\end{equation}
\end{proposition}
We call \eqref{topenumerativesolution} the {\it topological-enumerative solution} of \eqref{odedubrovin} .

In \cite{Galkin2014gammaclasses}, Galkin, Golyshev and Iritani gave an analytic criterion to obtain the above calibration when $X$ is Fano. 
\begin{proposition}[\cite{Galkin2014gammaclasses}, cf. \cite{Cotti2020LocalModuli, cotti2018helix}]\label{topologicalenumerativecriterion}
Let $X$ be a Fano manifold and fix a point ${\bf v}$ on the small quantum cohomology of $X$ (i.e. point with coordinates $v^i=0$ unless $i=2,\dots, r+1$). Then among all solutions of the form
\[\Phi(z)z^\mu z^{R}\] to \eqref{odedubrovin} around $z=0$, the topological-enumerative solution \eqref{topenumerativesolution} is the unique one for which the product
\begin{equation}
    H(z)=z^{-\mu}\Phi(z)z^\mu,
\end{equation} 
is holomorphic at $z=0$ and satisfies
\begin{equation}
    H(0)=\exp\left(\left(\sum_{i=2}^{r+1} v^i \phi_i\right)\cup(-)\right).
\end{equation}
\end{proposition}

\section{Exceptional collections, Gamma classes and a refined version of the Dubrovin conjecture}\label{sec:conjecture} Let $X$ be a smooth projective variety of complex dimension $D$ with odd-vanishing cohomology, and denote by $\mathcal{D}^b(X)$ the bounded derived category of coherent sheaves on $X$.  Given $E,F\in{\rm Ob}\left(\mathcal{D}^b(X)\right)$, define $\text{Hom}^*(E,F)$ as the $\mathbb{C}$-vector space
\begin{equation*}
\text{Hom}^*(E,F):=\bigoplus_{k\in\mathbb{Z}}\text{Hom}(E,F[k]).
\end{equation*}
A collection $(E_1,\dots, E_n)$ of objects of $\mathcal{D}^b(X)$ is called an \emph{exceptional collection} if 
\begin{enumerate}
\item $\text{Hom}^*(E_k,E_k)$ is a one dimensional $\mathbb{C}$-algebra for each $1\leq k \leq n$, generated by the identity morphism,
\item we have $\text{Hom}^*(E_j,E_k)=0$ for $j>k$.
\end{enumerate}
An exceptional collection is \emph{full} if any triangular subcategory containing all objects of this collection is equivalent to $\mathcal{D}^b(X)$ via the inclusion functor.

Now let $TX$ be the tangent bundle of $X$, and let $\delta_1,\dots,\delta_D$ be the Chern roots of $TX$, so that 
\[c_k(TX)=\sigma_k(\delta_1,\dots,\delta_r),\quad 1\leq k\leq D, 
\]where $\sigma_k$ is the $k$-th elementary symmetric polynomial.
The \emph{Gamma classes} of $X$ are the characteristic classes (cf. \cite{hirzebruch1966topological})
\begin{equation}
\widehat{\Gamma}_X^\pm:=\prod_{j=1}^D\Gamma(1\pm\delta_j), 
\end{equation}
where $\gamma$ is the Euler-Mascheroni constant and $\zeta$ is the Riemann zeta function.

Consider an arbitrary vector bundle $V$ on $X$, and let $\tau_1,\dots,\tau_r$ be the Chern roots of $V$. 
The {\it Chern character} of $V$ is the characteristic class ${\rm Ch}(V)\in H^*(X)$ defined by ${\rm Ch}(V):=\sum_{j=1}^r e^{2\pi i \tau_j}$.

The following conjecture was proposed by Dubrovin in his 1998 ICM talk.
\begin{conjecture}[\cite{Dub98Geometry}]\label{originalconjecture} Assume that $X$ is Fano and that the power series $F_0(\textbf{v})$ has a convergence domain $\Omega$. Then
\begin{enumerate}
    \item [1.] the quantum cohomology of $X$ is semisimple if and only if $\mathcal{D}^b(X)$ admits a full exceptional collection. 
\end{enumerate}
If the quantum cohomology of $X$ is semisimple, then there exists a full exceptional collection $(E_1,\dots, E_n)$ of $\mathcal{D}^b(X)$ such that: 
\begin{enumerate}
    \setcounter{enumi}{+1}
    \item [2.] the Stokes matrix $S$ is equal to the inverse of the Euler matrix $\left(\chi(E_j,E_k)\right)_{1\leq j,k \leq n}$;
    \item [3.] the central connection matrix has the form $C=C'C''$ where the columns of $C''$ are coordinates of ${\rm Ch}(E_i)$, and $C' \colon H^*(X;\mathbb C)\to H^*(X;\mathbb C) $ is an operator commuting with $c_1(X)\cup \colon H^*(X;\mathbb C)\to H^*(X;\mathbb C)$.
\end{enumerate}
\end{conjecture}

In \cite{Dubrovin13}, Dubrovin suggested that the matrix $C'$ should be given by using the Gamma class of $X$. In \cite{cotti2018helix}, Cotti, Dubrovin and Guzzetti proposed a refinement of Conjecture \ref{originalconjecture}. 

\begin{conjecture}[\cite{cotti2018helix}]\label{refinedconjecture}
    Assume that $X$ is a Fano variety of dimension $D$ with odd-vanishing cohomology and that the power series $F_0(\textbf{v})$ has a convergence domain $\Omega$. Then
\begin{enumerate}
    \item [1.] the quantum cohomology of $X$ is semisimple if and only if $\mathcal{D}^b(X)$ admits a full exceptional collection. 
\end{enumerate}
If the quantum cohomology of $X$ is semisimple, then there exists a full exceptional collection $(E_1,\dots, E_n)$ of $\mathcal{D}^b(X)$ such that: 
\begin{enumerate}
     \setcounter{enumi}{+1}
    \item [2.] the Stokes matrix $S$ is equal to the inverse of the Euler matrix $\left(\chi(E_j,E_k)\right)_{1\leq j,k \leq n}$;
    \item [3.] the central connection matrix $C$, connecting the solution $Y_{R}$ (see \eqref{YLRasymptotic}) with the topological-enumerative solution $Y_{\rm top}=\Psi\cdot \mathcal{Y}_{\rm top}$ of Proposition \ref{topologicalenumerativeprop}, is equal to the matrix whose columns are given by the coordinates of 
    \begin{equation*}
        \frac{i^{\bar D}}{(2\pi)^\frac{D}{2}}\widehat\Gamma^-_X\cup\exp(-\pi i c_1(X))\cup{\rm Ch}(E_k), \quad 1\leq k \leq n.  
    \end{equation*}
    Here $\bar D = D\ \text{mod}\ 2$. 
\end{enumerate}
\end{conjecture}

\begin{remark}
      In \cite{cotti2018helix}, Cotti, Dubrovin and Guzzetti proved that part 3 of Conjecture \ref{refinedconjecture} is equivalent to the Gamma conjecture II proposed by Galkin, Golyshev and Iritani \cite{Galkin2014gammaclasses}.
\end{remark}

\section{Proof of the refined Dubrovin conjectrue for $LG(2,4)$}\label{sec:4}

\subsection{Small quantum cohomology of $LG(2,4)$}
It is known that the Lagrangian Grassmanian $LG(2,4)$ is isomorphic to the 3-dimensional quadric. 
The cohomology ring $H^*(LG(2,4);\mathbb{C})$ has the following ring presentation (cf. \cite{borel1953cohomologie})
\begin{equation}\label{LG24cohomology}
    H^*(LG(2,4);\mathbb{C}) \cong \mathbb{C}[x_1,x_2]/\left\langle x_1^2 - 2 x_2,  x_2^2 \right\rangle.
\end{equation}

From the general theory of Schubert calculus \cite{bernstein1973schubert}, it is known that $H^*(LG(2,4);\mathbb{C})$ is a complex linear space of dimension 4 and a basis is given by {\it Schubert classes}: 
\begin{equation}\label{LG24schubertclass}
    \sigma_0:=1,  \sigma_1,  \sigma_2,  \sigma_{2,1}. 
\end{equation}
We denote them by $e_1:=\sigma_0, e_2:=\sigma_1, e_3:=\sigma_2, e_4:=\sigma_{2,1}$, and denote by $v^i$ the coordinate with respect to $e_i$. Then the coordinates in the small quantum cohomology are
\begin{equation*}
    \textbf{v}=(0,v^2,0,0). 
\end{equation*}
The corresponding {\it Schubert polynomials} are
    \begin{equation}\label{LG24Schubertpoly}
    \begin{aligned}
        &\mathfrak{S}_{0} = 1 , 
     \\ &\mathfrak{S}_{1} = x_1 ,\  \mathfrak{S}_{2}= \frac{1}{2}x_1^2 , 
     \\ &\mathfrak{S}_{2,1}=\frac{1}{2}x_1^3, 
    \end{aligned}
    \end{equation}
which are images of the Schubert classes under the isomorphism \eqref{LG24cohomology}. 
The first Chern class of $LG(2,4)$ is $c_1=3\sigma_1$. The matrix of the multiplication $c_1\cup$ under the basis \eqref{LG24schubertclass} is
\begin{equation}
   R = \left(
\begin{array}{cccc}
 0 & 0 & 0 & 0 \\
 3 & 0 & 0 & 0 \\
 0 & 6 & 0 & 0 \\
 0 & 0 & 3 & 0 \\
\end{array}
\right). 
\end{equation}
The matrix of Poincar{\'e} pairing 
\begin{equation}
    \eta(\alpha,\beta):=\int_{LG(2,4)}\alpha \wedge \beta
\end{equation}
is 
\begin{equation}
   \eta = \left(
\begin{array}{cccc}
 0 & 0 & 0 & 1 \\
 0 & 0 & 1 & 0 \\
 0 & 1 & 0 & 0 \\
 1 & 0 & 0 & 0 \\
\end{array}
\right). 
\end{equation}

The small quantum cohomology ring $qH^*(LG(2,4);\mathbb{C})$ has the following ring presentation \cite{kresch2003quantum}: 
\begin{equation}\label{LG24quantumcohomology}
    qH^*(LG(2,4);\mathbb{C}) \cong \mathbb{C}[x_1,x_2,q]/\left\langle x_1^2 - 2 x_2, -q x_1 + x_2^2 \right\rangle, 
\end{equation}
where $q=e^{v^2}$. 
Images of the Schubert classes under the isomorphism \eqref{LG24quantumcohomology} are \cite{kresch2003quantum}
    \begin{equation}\label{LG24qSchubertpoly}
    \begin{aligned}
        &\mathfrak{S}_{0}^q = 1 , 
     \\ &\mathfrak{S}_{1}^q = x_1 ,\  \mathfrak{S}_{2}^q= \frac{1}{2}x_1^2 , 
     \\ &\mathfrak{S}_{2,1}^q=\frac{1}{2}x_1^3-q. 
    \end{aligned}
    \end{equation}

The multiplication table of the small quantum cohomology of $LG(2,4)$ is shown in \autoref{multitable}. 
\begin{table}[]
\renewcommand\arraystretch{1.3}
\begin{tabular}{|c|c|c|c|}
\hline
$\cdot$        & $\sigma_1$        & $\sigma_2$        & $\sigma_{2,1}$ \\ \hline
$\sigma_1$     & $2 \sigma_2$      & $q+ \sigma _{2,1}$ & $q \sigma_1$   \\ \hline
$\sigma_2$     & $q+\sigma _{2,1}$ & $q \sigma_1$      & $q \sigma_2$   \\ \hline
$\sigma_{2,1}$ & $q \sigma_1$      & $q \sigma_2$      & $q^2$          \\ \hline
\end{tabular}
\vspace{5pt}
\caption{\label{multitable}Multiplication table of the small quantum cohomology of $LG(2,4)$}
\end{table}


It is known that $q=1$ (i.e. $v^i=0$, $1\leq i \leq 4$) is a semisimple point of the small quantum cohomology of $LG(2,4)$ \cite{Perrin2010quantumminusculeIII}. Then according to a result of Cotti \cite{cotti2021degenerate} we know that $F_0(\bf v)$ (see \eqref{F0v}) has a convergence domain.


\subsection{Deformed flat connection}


 At the point $q=1$, equation \eqref{odedubrovin} becomes
\begin{equation}\label{LG24ODEsystem}
    \begin{aligned}
    \dfrac{dy}{dz}=\left(\mathcal{U}+\dfrac{\mu}{z}\right)y, 
    \end{aligned}
\end{equation}
where
\begin{equation}
   \mathcal{U} = \left(
\begin{array}{cccc}
 0 & 0 & 3 & 0 \\
 3 & 0 & 0 & 3 \\
 0 & 6 & 0 & 0 \\
 0 & 0 & 3 & 0 \\
\end{array}
\right)
\end{equation} 
and 
\begin{equation}
   \mu = \left(
\begin{array}{cccc}
 -\frac{3}{2} & 0 & 0 & 0 \\
 0 & -\frac{1}{2} & 0 & 0 \\
 0 & 0 & \frac{1}{2} & 0 \\
 0 & 0 & 0 & \frac{3}{2} \\
\end{array}
\right).
\end{equation}

\begin{proposition}\label{LG24ODEsystemProp}
    The system \eqref{LG24ODEsystem} can be reduced to an equivalent scalar ODE
    \begin{equation}\label{scalarODE}
    D^4\varphi-108  z^3 D\varphi-162 z^3\varphi=0, 
    \end{equation}
    where $D = z \frac{d}{dz}$.
\end{proposition}
\begin{proof}
    Let $y=(y_1,\dots,y_4)^T$ and $y_4(z)=z^{\frac{3}{2}} \varphi(z)$. From the system \eqref{LG24ODEsystem}, $y_1,y_2,y_3$ are uniquely determined by $\varphi$ through the following equations: 
    \begin{equation*}
    \begin{aligned}
     y_1(z)&=\frac{z^2 \varphi^{(3)}(z)+\varphi'(z)+3 z \varphi''(z)-54 z^2 \varphi(z)}{54 \sqrt{z}},
\\    y_2(z)&=\frac{1}{18} \left(z^{\frac{3}{2}} \varphi''(z)+\sqrt{z} \varphi'(z)\right),
\\    y_3(z)&=\frac{1}{3} z^{\frac{3}{2}} \varphi'(z),
\end{aligned}
\end{equation*}
where $\varphi$ satisfies the scalar ODE \eqref{scalarODE}. 
\end{proof}

The indicial equation of \eqref{scalarODE} at $z=0$ reads
\begin{equation}
    r^4=0. 
\end{equation}
A solution of \eqref{scalarODE} is (cf. \cite{coates2016quantumperiods})
\begin{equation}
    \tilde{\varphi}(z)= \sum^{\infty}_{d=0} \dfrac{(2d)!}{(d!)^5}z^{3d}=1 + 2 z^3 + \frac {3} {4}z^6 + \frac {5 } {54}z^9 + \frac {35 } {6912}z^{12} + \frac {7 } {48000}z^{15}+O(z^{18}).\label{LG24quantumperiod}
\end{equation}

By Proposition \ref{topologicalenumerativeprop}, the topological-enumerative solution of \eqref{LG24ODEsystem} is 
\begin{equation}\label{leveltform}
    \mathcal{Y}_{\text{top}}(z) = \Phi_{\text{top}}(z) z^{\mu}z^{R},
\end{equation}
where 
\begin{equation}
\begin{aligned}
    \Phi_{\text{top}}(z) &= I +  \left(
\begin{array}{cccc}
 0 & 0 & 1 & 0 \\
 0 & 0 & 0 & 1 \\
 0 & 0 & 0 & 0 \\
 0 & 0 & 0 & 0 \\
\end{array}
\right) z+ \left(
\begin{array}{cccc}
 0 & -2 & 0 & 0 \\
 0 & 0 & 0 & 0 \\
 0 & 0 & 0 & 2 \\
 0 & 0 & 0 & 0 \\
\end{array}
\right) z^2+\left(
\begin{array}{cccc}
 2 & 0 & 0 & 1 \\
 0 & -2 & 0 & 0 \\
 0 & 0 & -2 & 0 \\
 0 & 0 & 0 & 2 \\
\end{array}
\right) z^3
\\ &+\left(
\begin{array}{cccc}
 0 & 0 & -\frac{3}{2} & 0 \\
 4 & 0 & 0 & \frac{3}{2} \\
 0 & 0 & 0 & 0 \\
 0 & 0 & -4 & 0 \\
\end{array}
\right)z^4
+\left(
\begin{array}{cccc}
 0 & \frac{3}{2} & 0 & 0 \\
 0 & 0 & -\frac{7}{2} & 0 \\
 8 & 0 & 0 & \frac{3}{2} \\
 0 & 8 & 0 & 0 \\
\end{array}
\right)z^5+\left(
\begin{array}{cccc}
 \frac{13}{4} & 0 & 0 & \frac{1}{2} \\
 0 & \frac{33}{4} & 0 & 0 \\
 0 & 0 & -\frac{17}{4} & 0 \\
 0 & 0 & 0 & \frac{3}{4} \\
\end{array}
\right)z^6
\\ &\qquad \qquad \qquad+\left(
\begin{array}{cccc}
 0 & 0 & -\frac{19}{12} & 0 \\
 -\frac{5}{2} & 0 & 0 & \frac{5}{12} \\
 0 & \frac{25}{2} & 0 & 0 \\
 0 & 0 & -\frac{5}{2} & 0 \\
\end{array}
\right)z^7+O(z^8). 
\end{aligned}
\end{equation}




\subsection{Computation of the Stokes matrix and the central connection matrix}
It suffices to compute the monodromy data at $q=1$ due to the simisimplicity. 
Let $\epsilon=e^{\frac{2}{3}i\pi}$. The canonical coordinates at $q=1$ are 
\begin{equation}\label{LG24eigenvalues}
    u_1=0,\ u_2=3\cdot 2^{\frac{2}{3}},\  u_3=3\cdot 2^{\frac{2}{3}}\epsilon^{2},\  u_4=3\cdot 2^{\frac{2}{3}}\epsilon.  
\end{equation}
The Stokes rays (see \eqref{stokesrays}) are 
\begin{equation*}
\begin{aligned}
        R_{12}&=\left\{i\rho: \rho\geq0\right\},\ R_{13}=\left\{-\rho e^{i \frac{\pi}{6}}: \rho\geq0\right\},\ R_{14}=\left\{\rho e^{-i \frac{\pi}{6}}: \rho\geq0\right\}, 
    \\  R_{23}&=\left\{-\rho e^{i \frac{\pi}{3}}: \rho\geq0\right\},\ R_{24}=\left\{\rho e^{-i \frac{\pi}{3}}: \rho\geq0\right\},\ 
    \\  R_{34}&=\left\{\rho: \rho\geq0\right\}. 
\end{aligned}
\end{equation*}

\begin{figure}
    \centering
    \label{fig:enter-label}
\begin{tikzpicture}
  \filldraw[fill=gray!5!white, draw=gray!50!black] (0,0) -- (0.866025,0.5) arc [start angle=30, end angle=240, radius=1];
  \filldraw[fill=gray!40!white, draw=black!50!black] (0,0) -- (-1.29904,-0.75) arc [start angle=-150, end angle=60, radius=1.5];
  \filldraw[fill=gray!80!white, draw=black!50!black] (0,0) -- (1.73205,1) arc [start angle=30, end angle=60, radius=2]  node[near start,above=5pt] {$\Pi_{+}$};
  \filldraw[fill=gray!80!white, draw=black!50!black] (0,0) -- (-1.73205,-1) arc [start angle=-150, end angle=-120, radius=2] node[near start,below=5pt] {$\Pi_{-}$};
  \draw[->] (-2.5,0) -- (2.5,0);
  \draw[->] (0,-2.5) -- (0,2.5);
  \draw (-2.3,-2.3) -- (2.3,2.3) node[above] {$\ell$};
  \draw (-2.16506,-1.25) -- (2.16506,1.25)  node[below=5pt, font=\tiny] {$R_{31}$}  node[very near start, left=10pt,font=\tiny] {$R_{13}$} ;
  \draw (-1.25,-2.16506) -- (1.25,2.16506)  node[very near end, above=13pt,font=\tiny] {$R_{32}$} node[very near start, below=13pt,font=\tiny] {$R_{23}$};
  \draw[<->] (0.866025,0.5) arc [start angle=30, end angle=240, radius=1] node[midway,above=10pt] {$\Pi_{\text{left}}$};
  \draw[<->] (-1.29904,-0.75) arc [start angle=-150, end angle=60, radius=1.5] node[midway,below=10pt] {$\Pi_{\text{right}}$};
\end{tikzpicture}
  \caption{The admissilbe line, Stokes rays and sectors}
\end{figure}
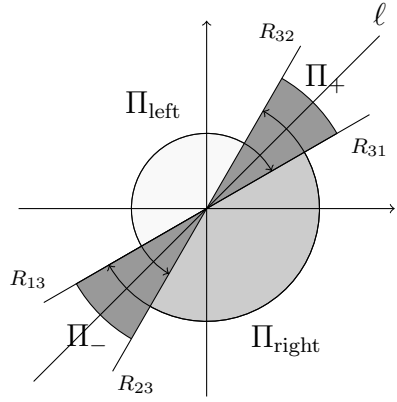

\noindent We fix the admissible line $\ell$: 
\begin{equation*}
     \ell=\left\{\rho e^{i \frac{\pi}{4}}: \rho\in \mathbb{R}\right\}. 
\end{equation*}
Let 
\begin{equation*}
    \Pi_{\text{left}}=\left\{z:\frac{\pi}{6} < \text{arg}z < \frac{4\pi}{3}\right\},\quad \Pi_{\text{right}}=\left\{z:-\frac{5\pi}{6}< \text{arg}z < \frac{\pi}{3}\right\}. 
\end{equation*}
The narrow sectors are then given by
\begin{equation*}
    \Pi_{+}=\left\{z:\frac{\pi}{6}< \text{arg}z <\frac{\pi}{3}\right\}, 
    \quad \Pi_{-}=\left\{z:-\frac{\pi}{6}< \text{arg}z <\frac{\pi}{3}\right\}.
\end{equation*}

\subsubsection{Solutions to the scalar ODE \eqref{scalarODE}}
\begin{lemma}
    The following functions are solutions of the ODE \eqref{scalarODE}: 
\begin{itemize}
    \item the function
\begin{equation}\label{sol1}
    \varphi_1(z)=\int_{\Lambda_1} \frac{\Gamma (s)^4}{\Gamma \left(s+\frac{1}{2}\right)}2^{-2 s} e^{\pi  i s} z^{-3 s} ds, \quad \ -\frac{\pi}{2}< \textup{arg} z<\frac{\pi}{2}, 
\end{equation}
where $\Lambda_1$ is any line in the complex plane from the point $\kappa - i \infty$ to $\kappa + i \infty$ for any $\kappa>0$; 

    \item the function 
\begin{equation}\label{sol2}
   \varphi_2(z)=\int_{\Lambda_2}  \Gamma (s)^4 \Gamma \left(\frac{1}{2}-s\right) 2^{-2 s} z^{-3 s} ds, \quad \ -\frac{5\pi}{6}< \textup{arg} z<\frac{5\pi}{6},
\end{equation}
where $\Lambda_2$ is any line in the complex plane from the point $\kappa - i \infty$ to $\kappa + i \infty$ for any $0<\kappa<\frac{1}{2}$. 
\end{itemize}
\end{lemma}
\begin{proof}
    This Lemma can be obtained via a direct verification (cf. \cite{dubrovin1999painleve}). 
\end{proof}

Now we consider asymptotic behaviours of $\varphi_1$ and $\varphi_2$ as $|z| \rightarrow \infty$ and several identities. 
\begin{lemma}
   The following asymptotic behaviours hold as $|z| \rightarrow \infty$ 
\begin{equation}
\begin{aligned}
    -\frac{z^{\frac{3}{2}}}{2\sqrt{2}\pi^2} \varphi_1(z)&= \frac{1}{\sqrt{6}} e^{u_4 z}\left(1+O\left(\dfrac{1}{z}\right)\right), \quad \textup{for}\ -\frac{\pi}{2}< \textup{arg} z<\frac{\pi}{2},
    \\   -\frac{z^{\frac{3}{2}}}{\sqrt{2} \pi^3} \varphi_2(z)&= -\frac{i}{\sqrt{2}} e^{u_1 z}\left(1+O\left(\dfrac{1}{z}\right)\right), \quad \textup{for}\ -\frac{\pi}{6}< \textup{arg} z<\frac{\pi}{2}.
\end{aligned}
\end{equation} 
\end{lemma}
\begin{proof}
   The proof of this Lemma is an elementary exercise by using Laplace method (cf. \cite{dubrovin1999painleve}).
\end{proof}

\begin{lemma}
   Let $\epsilon=e^{\frac{2}{3}i\pi}$. The functions $\varphi_1, \varphi_2$ satisfy the following identity
    \begin{equation}\label{euler}
           \varphi_2(z\epsilon^{-1})= 2\pi \varphi_1(z)-\varphi_2(z), \quad \textup{for}\ -\frac{\pi}{6}< \textup{arg} z<\frac{\pi}{2}. 
    \end{equation}
\end{lemma}
\begin{proof}
    This Lemma can be obtained via a direct verification (cf. \cite{Cotti2020LocalModuli}). 
\end{proof}

\begin{lemma}
    Any solution $\varphi$ to \eqref{scalarODE} satisfies the following identity
    \begin{equation}\label{rotation}
    \varphi(z\epsilon^4)-4\varphi(z\epsilon^3)+6\varphi(z\epsilon^2)-4\varphi(z\epsilon)+\varphi(z)=0. 
    \end{equation}
\end{lemma}
\begin{proof} 
(cf. \cite{dubrovin1999painleve}) Note that if $\varphi(z)$ is a solution to \eqref{scalarODE} then $\varphi(z\epsilon)$ also is. Then the map $z\mapsto z\epsilon$ defines a linear map on the solution space of \eqref{scalarODE} 
, where the solutions have the form
\begin{equation*}\label{phiz=0}
    \varphi(z)=\sum_{n\geq 0}z^{3n}\left(a_n+b_n\log z+c_n(\log z)^2+d_n(\log z)^3 \right),
\end{equation*}
where $a_0$, $b_0$, $c_0$, $d_0$ are arbitrary constants, and $a_n, b_n, c_n, d_n$, $n\geq 1$ can be obtained recursively. Choosing a basis of solutions with coefficients $(a_0,b_0,c_0,d_0)=(1,0,0,0)$, $(0,1,0,0)$ and $(0,0,1,0)$, $(0,0,0,1)$, the matrix of the operator
\begin{equation*}
    (A\varphi)(z):=\varphi(z\epsilon)
\end{equation*}
is triangular whose diagonal entires are all 1. By Cayley--Hamilton theorem we then deduce that
$ (A- I)^5=0$, namely
\begin{equation*}
\pushQED{\qed} 
    A^4-4A^3+6A^2-4A+I=0. \qedhere
\popQED
\end{equation*}\phantom\qedhere\end{proof}

\subsubsection{Asymptotic behaviour of fundamental matrix solutions}

Let $\Psi$ be the transition matrix from $(f_1,f_2,f_3,f_4)$  to $(e_1,e_2,e_3,e_4)$ at $q=1$, i.e. (see \eqref{psitransition})
\begin{equation}
    e_{\alpha}=\sum_{j=1}^4 \psi_{j \alpha} f_j , \qquad \alpha = 1,2,3,4. 
\end{equation}
We can compute directly 
\begin{equation}
\begin{aligned}
    f_1 &= \frac{i}{\sqrt{2}}e_1-\frac{i}{\sqrt{2}}e_4,
\\  f_2 &= \frac{1}{\sqrt{6}}e_1+\frac{1}{\sqrt[6]{2} \sqrt{3}}e_2+\frac{\sqrt[6]{2}}{\sqrt{3}}e_3+\frac{1}{\sqrt{6}}e_4,
\\  f_3 &= \frac{1}{\sqrt{6}}e_1+\frac{1}{\sqrt[6]{2} \sqrt{3}}e^{\frac{2 i \pi }{3}}e_2+\frac{\sqrt[6]{2} }{\sqrt{3}} e^{\frac{-2 i \pi }{3}} e_3+\frac{1}{\sqrt{6}}e_4, 
\\  f_4 &= \frac{1}{\sqrt{6}}e_1+\frac{1}{\sqrt[6]{2} \sqrt{3}}e^{\frac{-2 i \pi }{3}}e_2+\frac{\sqrt[6]{2}}{\sqrt{3}}e^{\frac{2 i \pi }{3}}e_3+\frac{1}{\sqrt{6}}e_4. 
\end{aligned}
\end{equation} 
Then
\begin{equation}
    \Psi = \left(
\begin{array}{cccc}
 -\frac{i}{\sqrt{2}} & 0 & 0 & \frac{i}{\sqrt{2}} \\
 \frac{1}{\sqrt{6}} & \frac{\sqrt[6]{2}}{\sqrt{3}} & \frac{1}{\sqrt[6]{2} \sqrt{3}} & \frac{1}{\sqrt{6}} \\
 \frac{1}{\sqrt{6}} & -\frac{\sqrt[6]{2}}{\sqrt{3}}e^{\frac{i \pi}{3}}  & \frac{1}{\sqrt[6]{2} \sqrt{3}}e^{\frac{2i \pi}{3}} & \frac{1}{\sqrt{6}} \\
 \frac{1}{\sqrt{6}} & \frac{\sqrt[6]{2}}{\sqrt{3}}e^{\frac{2i \pi}{3}}  & -\frac{1}{\sqrt[6]{2} \sqrt{3}}e^{\frac{i \pi}{3}} & \frac{1}{\sqrt{6}} \\
\end{array}
\right). 
\end{equation}
Denote 
\begin{equation}
   U= \Psi\mathcal{U}\Psi^{-1} = \left(
\begin{array}{cccc}
 0 & 0 & 0 & 0 \\
 0 & 3\cdot 2^{\frac{2}{3}} & 0 & 0 \\
 0 & 0 & 3\cdot 2^{\frac{2}{3}}\epsilon^{2} & 0 \\
 0 & 0 & 0 & 3\cdot 2^{\frac{2}{3}}\epsilon \\
\end{array}
\right), 
\end{equation}
and 
\begin{equation}
   V= \Psi\mu\Psi^{-1} = \left(
\begin{array}{cccc}
 0 & \frac{\sqrt{3}}{2}i  & \frac{\sqrt{3}}{2}i & \frac{\sqrt{3}}{2}i \\
 -\frac{\sqrt{3}}{2}i & 0 & -\frac{\sqrt{3}}{6}i & \frac{\sqrt{3}}{6}i \\
 -\frac{\sqrt{3}}{2}i & \frac{\sqrt{3}}{6}i & 0 & -\frac{\sqrt{3}}{6}i \\
 -\frac{\sqrt{3}}{2}i & -\frac{\sqrt{3}}{6}i & \frac{\sqrt{3}}{6}i & 0 \\
\end{array}
\right). 
\end{equation}
For a fundamental matrix solution $\mathcal{Y}(z)$ to \eqref{LG24ODEsystem}, let
\begin{equation}
    Y(z) = \Psi\mathcal{Y}(z). 
\end{equation}
Then $Y=Y(z)$ satisfies the equation
\begin{equation}\label{LG24ODEsystemnormalized}
    \begin{aligned}
    \dfrac{dY}{dz}=\left(U+\dfrac{V}{z}\right)Y. 
    \end{aligned}
\end{equation}

Note that there exists unique fundamental matrix solutions $Y_{L/R}$ to \eqref{LG24ODEsystemnormalized}, analytic in $\Pi_{\text{left/right}}$, respectively, such that 
\begin{equation}
    Y_{L/R}(z)\sim G(z) e^{zU} 
\end{equation}
as $|z| \rightarrow \infty$ within the sectors, where $G (z)$ is a formal power series of~$z^{-1}$ with $G(\infty)=I$ (see \eqref{asymptoticformal}). Then the corresponding solutions to \eqref{LG24ODEsystem} have the following asymptotic behaviour as $|z| \rightarrow \infty$
 \begin{equation}\label{LG24Yformal}
 \begin{aligned}
\mathcal{Y}_{L/R}(z) &= \Psi^{-1}Y_{L/R}(z) 
\\ &=\left(
\begin{array}{cccc}
 \frac{i}{\sqrt{2}}e^{u_1 z} & \frac{1}{\sqrt{6}}e^{u_2 z} & \frac{1}{\sqrt{6}}e^{u_3 z} & \frac{1}{\sqrt{6}}e^{u_4 z} \\
 0 & \frac{1}{\sqrt[6]{2} \sqrt{3}}e^{u_2 z} & -\frac{e^{-\frac{1}{3} i \pi}}{\sqrt[6]{2} \sqrt{3}}e^{u_3 z} & -\frac{e^{\frac{i \pi }{3}}}{\sqrt[6]{2} \sqrt{3}}e^{u_4 z} \\
 0 & \frac{\sqrt[6]{2}}{\sqrt{3}}e^{u_2 z} & -\frac{i \sqrt[6]{2} e^{-\frac{1}{6} i \pi}}{\sqrt{3}}e^{u_3 z} & -\frac{\sqrt[6]{2} e^{-\frac{1}{3} i \pi}}{\sqrt{3}}e^{u_4 z} \\
 -\frac{i}{\sqrt{2}}e^{u_1 z} & \frac{1}{\sqrt{6}}e^{u_2 z} & \frac{1}{\sqrt{6}}e^{u_3 z} & \frac{1}{\sqrt{6}}e^{u_4 z} \\
\end{array}
\right)   \left(1+O\left(\dfrac{1}{z}\right)\right).
\end{aligned}
\end{equation}


\subsubsection{Fundamental matrix solutions $\mathcal{Y}_{L/R}$ and the Stokes matrix}
Let $S'=(s'_{jk})$ be the Stokes matrix relating $Y_L$ and $Y_R$ in the narrow sector $\Pi_{+}$, i.e. 
\begin{equation}
    Y_L(z)=Y_R(z) S', \quad  z \in \Pi_{+}. 
\end{equation}
Then we have
 \begin{equation}\label{y44asm}
     y^L_{4k}(z)=\sum_{j=1}^4 y^R_{4j}(z) s'_{jk}, \ \ \text{for}\ \dfrac{\pi}{6}< \text{arg} z<\dfrac{\pi}{3},
 \end{equation}
Note that $y^{L/R}_{4k}$ have the following asymptotic behaviour as $|z| \rightarrow \infty$
\begin{equation}
    y^{L/R}_{4k}(z)=c_k e^{u_k z}\left(1+O\left(\dfrac{1}{z}\right)\right) ,\quad \text{for}\  z \in \Pi_{\text{left/right}}, 
\end{equation}
where $c_1=-\frac{i}{\sqrt{2}} ,\ c_2=\frac{1}{\sqrt{6}} ,\ c_3=\frac{1}{\sqrt{6}} ,\ c_4=\frac{1}{\sqrt{6}} $. 

\vspace{6pt}
By comparing asymptotic behaviours , we can determine the form of $S'$. For example, by definition, we have 
 \begin{equation}
     y^L_{44}(z)=\sum_{j=1}^4 y^R_{4j}(z) s'_{j4}, \quad \text{for}\ \dfrac{\pi}{6}< \text{arg} z<\dfrac{\pi}{3}.
 \end{equation}
Since $e^{u_4 z}$ is dominated by $e^{u_1 z},e^{u_2 z},e^{u_3 z}$ on the sector $0< \text{arg} z<\frac{2\pi}{3}$, i.e. $|e^{u_4 z}|<|e^{u_k z}|, \ k=1,2,3$, functions $y^R_{41}, y^R_{42}, y^R_{43}$ will not appear on the right-hand side. 
Then we have 
\begin{equation}\label{y44L=y44R}
    y^L_{44}(z)=y^R_{44}(z), \quad \text{for}\ \dfrac{\pi}{6}< \text{arg} z<\dfrac{\pi}{3}, 
\end{equation}
which means $s'_{14}=0, s'_{24}=0, s'_{34}=0, s'_{44}=1$. By similar dominance arguments, we have $s'_{11}=1, s'_{22}=1, s'_{33}=1, s'_{23}=0, s'_{21}=0, s'_{31}=0$ . Then $S'$ must have the form
\begin{equation}
    \left(
\begin{array}{cccc}
 1 & * & * & 0 \\
 0 & 1 & 0 & 0 \\
 0 & * & 1 & 0 \\
 * & * & * & 1 \\
\end{array}
\right).
\end{equation}

To explicitly compute $S'$, we have to construct entries of fundamental matrix solutions $\mathcal{Y}_{L/R}=(y^{L/R}_{jk})_{1\leq j,k \leq 4}$ on each sector $\Pi_{\text{left/right}}$ that have the same asymptotic behaviour as \eqref{LG24Yformal}. Note that we only have to construct $y^{L/R}_{4k}, \ 1\leq k \leq 4$  due to Proposition \ref{LG24ODEsystemProp}.

From \eqref{y44L=y44R} and the existence of $y^L_{44}$ and $ y^R_{44}$, $\varphi_1(z)$ can be analytically continued to the sector $-\frac{5\pi}{6}< \text{arg} z<\frac{4\pi}{3}$. Since the asymptotic behaviours of $y^L_{44}, y^R_{44}$ as $|z| \rightarrow \infty$ are same, we have 
\begin{equation}\label{phi1asym1}
    -\frac{z^{\frac{3}{2}}}{2\sqrt{2}\pi^2} \varphi_1(z)= \frac{1}{\sqrt{6}} e^{u_4 z}\left(1+O\left(\dfrac{1}{z}\right)\right),\quad \text{for}\  -\frac{5\pi}{6}< \text{arg} z<\frac{4\pi}{3}.
\end{equation}
Note that the continued  $\varphi_1(z)$ may not have the integral representation \eqref{sol1} in the above region. 
Since $\varphi_1(z\epsilon^{2})$ is also a solution of \eqref{scalarODE}, we have
\begin{equation}
    -\frac{z^{\frac{3}{2}}}{2\sqrt{2}\pi^2} \varphi_1(z\epsilon^{2}) = \text{linear combination of}\  y^R_{4k},\ 1\leq k \leq 4, \quad \text{for}\  -\frac{5\pi}{6}< \text{arg} z<0. 
\end{equation}
 Rotating $z$ by $\epsilon^{2}$ in \eqref{phi1asym1}, we get 
\begin{equation}\label{asmphi1q}
     -\frac{z^{\frac{3}{2}}}{2\sqrt{2}\pi^2} \varphi_1(z\epsilon^{2}) = \frac{1}{\sqrt{6}} e^{u_2 z}\left(1+O\left(\dfrac{1}{z}\right)\right),\quad \text{for}\  -\frac{13\pi}{6}< \text{arg} z<0. 
\end{equation}
From \eqref{asmphi1q} and the fact $e^{u_2 z}$ is dominated by $e^{u_1 z},e^{u_3 z},e^{u_4 z}$ on the sector $-\frac{4\pi}{3}< \text{arg} z<-\frac{2\pi}{3}$, we have 

\begin{equation}
     -\frac{z^{\frac{3}{2}}}{2\sqrt{2}\pi^2} \varphi_1(z\epsilon^{2}) = y^R_{42}(z),\quad \text{for}\    -\frac{5\pi}{6}< \text{arg} z<0. 
\end{equation}
By the existence of $y^R_{42}$ on $-\frac{5\pi}{6}< \text{arg} z<\frac{\pi}{3}$, $-\frac{z^{\frac{3}{2}}}{2\sqrt{2}\pi^2} \varphi_1(z\epsilon^{2})$ can be analytically continued to $-\frac{13\pi}{6}< \text{arg} z<\frac{\pi}{3}$ and still have the asymptotic behaviour as in \eqref{asmphi1q}. Then we have the following proposition for $\varphi_1(z)$. 
\begin{proposition}
   The following asymptotic behaviour holds as $|z| \rightarrow \infty$ 
\begin{equation}
    -\frac{z^{\frac{3}{2}}}{2\sqrt{2}\pi^2} \varphi_1(z)= \frac{1}{\sqrt{6}} e^{u_4 z}\left(1+O\left(\dfrac{1}{z}\right)\right),\quad \textup{for}\  -\frac{5\pi}{6}< \textup{arg} z<\frac{5\pi}{3}.
\end{equation}
\end{proposition}
\noindent Therefore, we conclude that 
\begin{equation}
    y^R_{42}(z) = -\frac{z^{\frac{3}{2}}}{2\sqrt{2}\pi^2} \varphi_1(z\epsilon^{2}),\quad \text{for}\  z \in \Pi_{\text{right}}.
\end{equation}

\vspace{15pt}

Now $\varphi_2(z)$ can be analytically continued to the sector $\frac{\pi}{2}< \text{arg} z<\frac{5\pi}{6}$ through the identity \eqref{euler}. From the fact $e^{u_4 z}$ is dominated by $e^{u_1 z}$ on the whole sector $-\frac{\pi}{6}< \text{arg} z<\frac{5\pi}{6}$ and the uniformity of asymptotic behaviour, we can determine the asymptotic behaviour of $\varphi_2(z)$ as $|z| \rightarrow \infty$: 
\begin{equation}
   -\frac{\sqrt{2}i z^{\frac{3}{2}}}{\pi^2} \varphi_2(z)= -\frac{i}{\sqrt{2}} e^{u_1 z}\left(1+O\left(\dfrac{1}{z}\right)\right),\quad \text{for}\  -\frac{\pi}{6}< \text{arg} z<\frac{5\pi}{6}.
\end{equation}
By the identity \eqref{euler} with the variable $z$ rotated by $\epsilon$ and similar arguments as above, we finally arrive at the following proposition for $\varphi_2(z)$.

\begin{proposition}
   The following asymptotic behaviour holds as $|z| \rightarrow \infty$ 
\begin{equation}\label{phi2asym}
   -\frac{\sqrt{2}i z^{\frac{3}{2}}}{\pi^2} \varphi_2(z)= -\frac{i}{\sqrt{2}} e^{u_1 z}\left(1+O\left(\dfrac{1}{z}\right)\right),\quad \textup{for}\  -\frac{5\pi}{6}< \textup{arg} z<\frac{5\pi}{6}.
\end{equation}
\end{proposition}

By the identity \eqref{euler}, $\varphi_2(z)$ can be further continued to $-\frac{3\pi}{2}< \text{arg} z<\frac{5\pi}{3}$, although it may not have the same asymptotic behaviour as \eqref{phi2asym}  and the integral representation as \eqref{sol2} beyond the sector $-\frac{5\pi}{6}< \text{arg} z<\frac{5\pi}{6}$. Since it is a linear combination of solutions, the continued $\varphi_2(z)$ is still a solution to \eqref{scalarODE}. Now the identity \eqref{euler} also holds in a larger sector. 
\begin{lemma}\label{LG24identity}
    The identity \eqref{euler} holds for $-\frac{5\pi}{6}< \textup{arg} z<\frac{5\pi}{3}$.
\end{lemma}

Now we can determine some entries of fundamental matrix solutions besides $y^R_{42}$: 
\begin{equation}\label{yLR44yL42yR43}
\begin{aligned}
    y^L_{44}(z)&=-\frac{z^{\frac{3}{2}}}{2\sqrt{2}\pi^2} \varphi_1(z),\quad \text{for}\  z \in \Pi_{\text{left}},  
 \\ y^R_{44}(z)&=-\frac{z^{\frac{3}{2}}}{2\sqrt{2}\pi^2} \varphi_1(z),\quad \text{for}\  z \in \Pi_{\text{right}},
 \\ y^L_{42}(z)&=\frac{z^{\frac{3}{2}}}{2\sqrt{2}\pi^2} \varphi_1(z\epsilon^{-1}),\quad \text{for}\  z \in \Pi_{\text{left}},
 \\ y^R_{43}(z)&=\frac{z^{\frac{3}{2}}}{2\sqrt{2}\pi^2} \varphi_1(z\epsilon),\quad \text{for}\  z \in \Pi_{\text{right}},
 \\ y^L_{41}(z)&=\frac{\sqrt{2}i z^{\frac{3}{2}}}{\pi^2} \varphi_2(z\epsilon^{-1}),\quad \text{for}\  z \in \Pi_{\text{left}},
 \\ y^R_{41}(z)&=-\frac{\sqrt{2}i z^{\frac{3}{2}}}{\pi^2} \varphi_2(z), \quad \text{for}\  z \in \Pi_{\text{right}}.
\end{aligned}
\end{equation}

 The only remained entry is $y^L_{43}$. For $y^L_{43}$, let some coefficients be undetermined: 
\begin{equation}\label{asmphi1q3}
    -\frac{z^{\frac{3}{2}}}{2\sqrt{2}\pi^2} \varphi_1(z\epsilon^{-2})=y^L_{43}(z)+\gamma_2''y^L_{42}(z),\quad \text{for}\   \frac{\pi}{2}< \text{arg} z<\frac{4\pi}{3},
\end{equation}
and 
\begin{equation}\label{asmphi1q4}
    \frac{z^{\frac{3}{2}}}{2\sqrt{2}\pi^2} \varphi_1(z\epsilon)=y^L_{43}(z)+\gamma_1''y^L_{41}(z)+\gamma_4''y^L_{44}(z),\quad \text{for}\  \frac{\pi}{6}< \text{arg} z<\pi.
\end{equation}
Subtract \eqref{asmphi1q3} and \eqref{asmphi1q4}, we get
\begin{equation}\label{subtract}
    -\frac{z^{\frac{3}{2}}}{2\sqrt{2}\pi^2} \varphi_1(z\epsilon^{-2})-\frac{z^{\frac{3}{2}}}{2\sqrt{2}\pi^2} \varphi_1(z\epsilon)=\gamma_2''y^L_{42}(z)-\gamma_1''y^L_{41}(z)-\gamma_4''y^L_{44}(z) ,\quad \text{for}\  \frac{\pi}{2}< \text{arg} z<\pi.
\end{equation}
We then replace all $\varphi_1$ by $\varphi_2$ in \eqref{subtract} by using \eqref{euler} and apply \eqref{rotation} to $\varphi_2$. By comparing cofficients, we obtain  
\begin{equation}
    y^L_{43}(z)=-\frac{z^{\frac{3}{2}}}{2\sqrt{2}\pi^2} \varphi_1(z\epsilon^{-2})+5y^L_{42}(z),\quad \text{for}\  \frac{\pi}{6}< \text{arg} z<\frac{2\pi}{3},
\end{equation}
and 
\begin{equation}
    y^L_{43}(z)=\frac{z^{\frac{3}{2}}}{2\sqrt{2}\pi^2} \varphi_1(z\epsilon)-4y^L_{41}(z)+5y^L_{44}(z),\quad \text{for}\  \frac{\pi}{2}< \text{arg} z<\frac{4\pi}{3}.
\end{equation}

Note that we have finished the construction of $\mathcal{Y}_{L/R}$. Comparing $\mathcal{Y}_{L}$ and $\mathcal{Y}_{R}$, we get the Stokes matrix $S'$ 
\begin{equation}
  S' = \left(
\begin{array}{cccc}
 1 & 4 & 4 & 0 \\
 0 & 1 & 0 & 0 \\
 0 & 5 & 1 & 0 \\
 -4 & -5 & -11 & 1 \\
\end{array}
\right) . 
\end{equation}
We then triangularize $S'$ by a permutation $P$ of canonical coordinates: 
\begin{equation}
S = PS'P^{-1}= \begin{pmatrix}
 1 & -4 & -11 & -5 \\
 0 & 1 & 4 & 4 \\
 0 & 0 & 1 & 5 \\
 0 & 0 & 0 & 1 \\
\end{pmatrix}, 
\end{equation}
where 
\begin{equation*}
  P = \left(
\begin{array}{cccc}
 0 & 0 & 0 & 1 \\
 1 & 0 & 0 & 0 \\
 0 & 0 & 1 & 0 \\
 0 & 1 & 0 & 0 \\
\end{array}
\right).  
\end{equation*}

\subsubsection{Central connection matrix}

Let $C^{'}=(c_{jk}^{'})_{1\leq j,k \leq 4}$ be the central connection matrix relating $\mathcal{Y}_{top}(z)$ and $\mathcal{Y}_R(z)$, i.e. 
\begin{equation}\label{centralconnetioneq}
   \mathcal{Y}_R(z) =\mathcal{Y}_{\text{top}}(z)C^{'} ,
\end{equation}

To compute $C^{'}$, we have to find the expansion of $Y^{R}$ at $z=0$ and compare the coefficents to determine $c_{jk}^{'}$. For example, 
\begin{equation}
\begin{aligned}
    y_{44}^R(z)&=-\frac{z^{\frac{3}{2}}}{2\sqrt{2}\pi^2} \varphi_1(z)=-\frac{z^{\frac{3}{2}}}{2\sqrt{2}\pi^2}\int_{\Lambda_1} \frac{\Gamma (s)^4}{\Gamma \left(s+\frac{1}{2}\right)}2^{-2 s} e^{\pi  i s} z^{-3 s} ds \\ &= -\frac{z^{\frac{3}{2}}}{2\sqrt{2}\pi^2}2\pi i\sum_{n=0}^{\infty}\text{res}_{s=-n}\left(\frac{\Gamma (s)^4}{\Gamma \left(s+\frac{1}{2}\right)}2^{-2 s} e^{\pi  i s} z^{-3 s}\right)
    \\  &=z^{\frac{3}{2}} \left(\frac{9 i }{2 \sqrt{2} \pi ^{3/2}}(\log z)^3+\frac{i (162 \gamma -54 i \pi ) }{12 \sqrt{2} \pi ^{3/2}}(\log z)^2 +\frac{i \left(162 \gamma ^2-108 i \gamma  \pi -15 \pi ^2\right) }{12 \sqrt{2} \pi ^{3/2}}\log z\right.
    \\ &\left.\quad \quad \quad +\frac{i \left(-12 \zeta (3)+54 \gamma ^3-54 i \gamma ^2 \pi -15 \gamma  \pi ^2+i \pi ^3\right)}{12 \sqrt{2} \pi ^{3/2}}\right)+z^{\frac{9}{2}} \left(\frac{9 i}{\sqrt{2} \pi ^{3/2}}(\log z)^3\right.
    \\ &\left.  +\frac{i (-216+324 \gamma -108 i \pi ) }{12 \sqrt{2} \pi ^{3/2}}(\log z)^2+\frac{i \left(144-432 \gamma +324 \gamma ^2+144 i \pi -216 i \gamma  \pi -30 \pi ^2\right) }{12 \sqrt{2} \pi ^{3/2}}\log z \right.
    \\ &\left.  +\frac{i \left(-24 \zeta (3)+144 \gamma -216 \gamma ^2+108 \gamma ^3-48 i \pi +144 i \gamma  \pi -108 i \gamma ^2 \pi +20 \pi ^2-30 \gamma  \pi ^2+2 i \pi ^3\right)}{12 \sqrt{2} \pi ^{3/2}}\right)
    \\ & \qquad \qquad \qquad \qquad \qquad \qquad \qquad \qquad  \qquad \qquad  \qquad \qquad \qquad \qquad  \qquad \qquad +O(z^{\frac{15}{2}}). 
\end{aligned}
\end{equation}
\\ From \eqref{centralconnetioneq},  
\begin{equation}
\begin{aligned}
    y_{44}^R&(z) = z^{\frac{3}{2}}  \left(9 c^{'}_{1,4} (\log z)^3+9 c^{'}_{2,4} (\log z)^2+3 c^{'}_{3,4} \log z+c^{'}_{4,4}\right)+z^{\frac{9}{2}} \left(18 c^{'}_{1,4} (\log z)^3\right.
   \\ &\left. +\left(18 c^{'}_{2,4}-36 c^{'}_{1,4}\right) (\log z)^2 +\left(24 c^{'}_{1,4}-24 c^{'}_{2,4}+6 c^{'}_{3,4}\right) \log z+8 c^{'}_{2,4}-4 c^{'}_{3,4}+2 c^{'}_{4,4}\right)+O(z^{15/2}). 
\end{aligned}
\end{equation}
Compare the coefficents we get 
\begin{equation}
\begin{aligned}
    c^{'}_{1,4} &= \frac{i}{2 \sqrt{2} \pi ^{3/2}},
\\  c^{'}_{2,4} &= \frac{\pi +3 i \gamma }{2 \sqrt{2} \pi ^{3/2}},
\\  c^{'}_{3,4} &= \frac{54 i \gamma ^2+36 \gamma  \pi -5 i \pi ^2}{12 \sqrt{2} \pi ^{3/2}},
\\  c^{'}_{4,4} &= -\frac{12 i \zeta (3)-54 i \gamma ^3-54 \gamma ^2 \pi +15 i \gamma  \pi ^2+\pi ^3}{12 \sqrt{2} \pi ^{3/2}}.
\end{aligned}
\end{equation}
In a similar way, we compare $y_{41}^R, y_{42}^R, y_{43}^R$ with right hand side of \eqref{centralconnetioneq} and we can obtain $C^{'}$. After the permutation $P$ of canonical coordinates, we get

\begin{equation}
\begin{aligned}
  C = C^{'}P^{-1} &=  \left(
\begin{array}{cc}
 \frac{i}{2 \sqrt{2} \pi ^{3/2}} & \frac{i}{\sqrt{2} \pi ^{3/2}}  \\
 \frac{\pi +3 i \gamma }{2 \sqrt{2} \pi ^{3/2}} & \frac{3 i \gamma }{\sqrt{2} \pi ^{3/2}}  \\
 \frac{54 i \gamma ^2+36 \gamma  \pi -5 i \pi ^2}{12 \sqrt{2} \pi ^{3/2}} & \frac{i \left(54 \gamma ^2+7 \pi ^2\right)}{6 \sqrt{2} \pi ^{3/2}}  \\
 -\frac{12 i \zeta (3)-54 i \gamma ^3-54 \gamma ^2 \pi +15 i \gamma  \pi ^2+\pi ^3}{12 \sqrt{2} \pi ^{3/2}} & \frac{i \left(-4 \zeta (3)+18 \gamma ^3+7 \gamma  \pi ^2\right)}{2 \sqrt{2} \pi ^{3/2}}  \\
\end{array}
\right.
\\ &\left.\quad \quad \begin{array}{cc}
  -\frac{i}{2 \sqrt{2} \pi ^{3/2}} & \frac{i}{2 \sqrt{2} \pi ^{3/2}} \\
\frac{\pi -3 i \gamma }{2 \sqrt{2} \pi ^{3/2}} & -\frac{3 (\pi -i \gamma )}{2 \sqrt{2} \pi ^{3/2}} \\
 \frac{-54 i \gamma ^2+36 \gamma  \pi +5 i \pi ^2}{12 \sqrt{2} \pi ^{3/2}} & \frac{i \left(54 \gamma ^2+108 i \gamma  \pi -53 \pi ^2\right)}{12 \sqrt{2} \pi ^{3/2}} \\
  \frac{12 i \zeta (3)+(\pi -3 i \gamma ) \left(18 \gamma ^2+12 i \gamma  \pi -\pi ^2\right)}{12 \sqrt{2} \pi ^{3/2}} & \frac{-4 i \zeta (3)+18 i \gamma ^3-54 \gamma ^2 \pi -53 i \gamma  \pi ^2+17 \pi ^3}{4 \sqrt{2} \pi ^{3/2}} \\
\end{array}
\right). 
\end{aligned}
\end{equation}

\subsection{Proof of the refined Dubrovin conjecture \ref{refinedconjecture} for $LG(2,4)$}

Let $\mathcal{U}$ be the tautological bundle on $LG(2,4)$, and $\mathcal{U}^*:= \sheafhom(\mathcal{U},\mathcal{O})$. The collection 
\begin{equation}\label{Fonarevcollection}
  \left( \mathcal{O}, \mathcal{O}(1), \Sigma^{(2,1)}\mathcal{U}^*, \mathcal{O}(2) \right)
\end{equation}
is a full exceptional collection \cite{fonarev2022FECLagrangian} (cf. \cite{Kapranov1988}), where $\Sigma$ is the Schur functor. Since the semisimplicity of the small quantum cohomology of $LG(2,4)$ at $q=1$ is known \cite{Perrin2010quantumminusculeIII}, we have proved part $1$ of Conjecture \ref{originalconjecture} for $LG(2,4)$.


Now twist \eqref{Fonarevcollection} by $\wedge^2 \mathcal{U}^*$: 
\begin{equation}\label{twistedcollection}
     \left( \mathcal{O}\otimes \wedge^2 \mathcal{U}^*, \mathcal{O}(1)\otimes \wedge^2 \mathcal{U}^*, \Sigma^{(2,1)}\mathcal{U}^*\otimes \wedge^2 \mathcal{U}^*, \mathcal{O}(2)\otimes \wedge^2 \mathcal{U}^* \right). 
\end{equation}
Denote the objects of \eqref{twistedcollection} by $E_k$, $ 1\leq k \leq 4$. The following proposition can be easily verified.   
\begin{proposition}
    The twisted collection \eqref{twistedcollection} is a full exceptional collection. 
\end{proposition}

\noindent The Euler matrix $\tilde{S} = (\chi(E_j,E_k))_{1\leq j,k \leq n}$ can be computed:   
\begin{equation}
   \tilde{S}= \left(
\begin{array}{cccc}
 1 & 5 & 16 & 14 \\
 0 & 1 & 4 & 5 \\
 0 & 0 & 1 & 4 \\
 0 & 0 & 0 & 1 \\
\end{array}
\right).
\end{equation}
The Gamma class $\hat{\Gamma}^{-}$ of $LG(2,4)$ is: 
\begin{equation}
 \hat{\Gamma}^{-}=1+3 \gamma  \sigma_1+\frac{1}{6} \left(54 \gamma ^2+\pi ^2\right) \sigma_2+\frac{1}{2}  \left(-4 \zeta (3)+18 \gamma ^3+\gamma  \pi ^2\right)\sigma_{2,1}. 
\end{equation}
The graded Chern characters are
\begin{equation}
\begin{aligned}
    \text{Ch}(\mathcal{O})&=1,
 \\ \text{Ch}(\mathcal{O}(1))&=1+2 i \pi  \sigma_1-4 \pi ^2 \sigma_2-\dfrac{8}{3}  i \pi ^3 \sigma_{2,1},
 \\ \text{Ch}(\Sigma^{(2,1)}\mathcal{U}^*)&=2+6 i \pi  \sigma_1-16 \pi ^2 \sigma_2-32 i \pi ^3 \sigma_{2,1},
 \\ \text{Ch}(\mathcal{O}(2))&=1+4 i \pi  \sigma_1-16 \pi ^2 \sigma_2-\dfrac{64}{3}  i \pi ^3 \sigma_{2,1}, 
 \\ \text{Ch}(\wedge^2 \mathcal{U}^*)&=1+2 i \pi  \sigma_1-4 \pi ^2 \sigma_2-\dfrac{8}{3}  i \pi ^3 \sigma_{2,1}. 
\end{aligned}
\end{equation}
Then we can compute the matrix  
\begin{equation}
\begin{aligned}
    C_{\Gamma}&= \left(
\begin{array}{cc}
 \frac{i}{2 \sqrt{2} \pi ^{3/2}} & \frac{i}{2 \sqrt{2} \pi ^{3/2}}  \\
 \frac{\pi +3 i \gamma }{2 \sqrt{2} \pi ^{3/2}} & -\frac{\pi -3 i \gamma }{2 \sqrt{2} \pi ^{3/2}}  \\
 \frac{54 i \gamma ^2+36 \gamma  \pi -5 i \pi ^2}{12 \sqrt{2} \pi ^{3/2}} & \frac{i \left(54 \gamma ^2+36 i \gamma  \pi -5 \pi ^2\right)}{12 \sqrt{2} \pi ^{3/2}}  \\
 -\frac{12 i \zeta (3)-54 i \gamma ^3-54 \gamma ^2 \pi +15 i \gamma  \pi ^2+\pi ^3}{12 \sqrt{2} \pi ^{3/2}} & \frac{-12 i \zeta (3)+54 i \gamma ^3-54 \gamma ^2 \pi -15 i \gamma  \pi ^2+\pi ^3}{12 \sqrt{2} \pi ^{3/2}}  \\
\end{array}
\right.
\\ &\quad \quad \quad \left.
\begin{array}{cc}
 \frac{i}{\sqrt{2} \pi ^{3/2}} & \frac{i}{2 \sqrt{2} \pi ^{3/2}} \\
\frac{-2 \pi +3 i \gamma }{\sqrt{2} \pi ^{3/2}} & -\frac{3 (\pi -i \gamma )}{2 \sqrt{2} \pi ^{3/2}} \\
 \frac{i \left(54 \gamma ^2+72 i \gamma  \pi -17 \pi ^2\right)}{6 \sqrt{2} \pi ^{3/2}} & \frac{i \left(54 \gamma ^2+108 i \gamma  \pi -53 \pi ^2\right)}{12 \sqrt{2} \pi ^{3/2}} \\
 \frac{2 \left(\pi ^3-6 i \zeta (3)\right)+54 i \gamma ^3-108 \gamma ^2 \pi -51 i \gamma  \pi ^2}{6 \sqrt{2} \pi ^{3/2}} & \frac{-4 i \zeta (3)+18 i \gamma ^3-54 \gamma ^2 \pi -53 i \gamma  \pi ^2+17 \pi ^3}{4 \sqrt{2} \pi ^{3/2}} \\
\end{array}
\right),
\end{aligned}
\end{equation}
whose columns are given by coordinates of
\begin{equation}
  \frac{i}{(2\pi)^{\frac{3}{2}}} \hat{\Gamma}^{-} \cup  e^{-i\pi c_1} \cup \text{Ch}(E_k),\quad 1\leq k \leq 4, 
\end{equation}
with respect to the basis $\sigma_0,  \sigma_1,  \sigma_2,  \sigma_{2,1}$.

\vspace{10pt}

\begin{theorem}\label{maintheorem}
    Conjecture \ref{refinedconjecture} holds for the Lagrangian Grassmanian $LG(2,4)$.
\end{theorem}
\begin{proof}
    Note that we have proved part $1$ of Conjecture \ref{originalconjecture} for $LG(2,4)$. Now consider the action of braid group $\mathcal{B}_4$, which is generated by 3 elementary braids $\beta_{12}, \beta_{23}, \beta_{34}$. The following sequence transforms $S$ into $\tilde{S}^{-1}$ and $C$ into $C_{\Gamma}$:
\begin{enumerate}
    \item The action by $\beta^{-1}_{23}$, 
    \item Change of signs by $\text{diag}(1,-1,-1,1)$.
\end{enumerate}    
    Therefore, we have proved part $2$ and $3$ of Conjecture \ref{originalconjecture} for $LG(2,4)$.  
\end{proof}
    
\smallskip
\noindent {\bf Acknowledgements}.
This work was done during my master thesis project. I would like to thank Di Yang for his advice, helpful discussions and encouragements. The work was supported by CAS No. YSBR-032, by National Key R and D Program of China 2020YFA0713100, and by NSFC No. 12371254.

\printbibliography

\end{document}